\documentclass[a4paper]{amsart}

\usepackage{amsmath}
\usepackage{amsfonts}
\usepackage{amsthm}
\usepackage{graphicx}
\usepackage{eucal}
\usepackage{amscd}
\usepackage[all,2cell]{xy}
\usepackage{amssymb}
\usepackage{mathrsfs}

\usepackage{tikz-cd}
\usetikzlibrary{matrix,arrows}
\newcommand{\Alg}{\mathsf{Alg}}
\newcommand{\Cal}[1]{{\mathcal #1}}
\newcommand{\End}{\operatorname{End}}

\newcommand{\Hom}{\operatorname{Hom}}
\newcommand{\Comm}{\operatorname{Comm}}
\newcommand{\AntiComm}{\operatorname{AntiComm}}
\newcommand{\Aut}{\operatorname{Aut}}
\newcommand{\ann}{\operatorname{ann}}

\def\Z{ \mathbb{Z} }
\DeclareMathOperator{\op}{op}

\newtheorem{thm}{Theorem}[section]

\newtheorem{cor}[thm]{Corollary}
\newtheorem{lem}[thm]{Lemma}
\newtheorem{prop}[thm]{Proposition}

\theoremstyle{definition}

\newtheorem{exams}[thm]{Examples}

\theoremstyle{remark}
\newtheorem{rem}[thm]{Remark}

\numberwithin{equation}{section}

\begin{document}

\title{Idempotent Pre-Endomorphisms of Algebras\\ }

\author{Fatma Azmy Ebrahim}
	\address{Department of Mathematics, Al-Azhar University, Egypt}
\email{fatema\texttt{\_}azmy@azhar.edu.eg}
\author{Alberto Facchini}
	\address{Dipartimento di Matematica ``Tullio Levi-Civita'', Universit\`a di 
Padova, 35121 Padova, Italy}
 \email{facchini@math.unipd.it}
\thanks{The second author is partially supported by Ministero dell'Universit\`a e della Ricerca (Progetto di ricerca di rilevante interesse nazionale ``Categories, Algebras: Ring-Theoretical and Homological Approaches (CARTHA)''), Fondazione Cariverona (Research project ``Reducing complexity in algebra, logic, combinatorics - REDCOM'' within the framework of the programme Ricerca Scientifica di Eccellenza 2018), and the Department of Mathematics ``Tullio Levi-Civita'' of the University of Padua (Research programme DOR1828909 ``Anelli e categorie di moduli'').}

\begin{abstract}
In the study of pre-Lie algebras, the concept of pre-morphism arises naturally as a generalization of the standard notion of morphism. Pre-morphisms can be defined for arbitrary (not-necessarily associative) algebras over any commutative ring $k$ with identity, and can be dualized in various ways to generalized morphisms (related to pre-Jordan algebras) and anti-pre-morphisms (related to anti-pre-Lie algebras).  We consider idempotent pre-endomorphisms (generalized endomorphisms, anti-pre-endomorphisms). Idempotent pre-endomorphisms are related to semidirect-product decompositions of the sub-adjacent anticommutative algebra.
\end{abstract}

\maketitle

\section{Introduction}

Idempotent endomorphisms of an algebraic structure are related to semidirect-product decompositions of the structure, actions and bimodules. For instance, for a vector space $V$, idempotent endomorphisms of $V$ are in a one-to-one correspondence with the pairs $(U,W)$, where $U$ and $W$ are subspaces of $V$ and $V$ is the internal direct sum $U\oplus W$. Similarly for an abelian group, or a module. For a group $G$, idempotent endomorphisms of $G$ are in a one-to-one correspondence with the pairs $(K,H)$, where $K$ is a normal subgroup of $G$, $H$ is a subgroup of $G$ and $G$ is the internal semidirect product $K\rtimes H$. In this case, an (external) semidirect product $K\rtimes H$ is completely described up to isomorphism by the action (a morphism) $H\to\Aut(K)$, that is to the $H$-group structure on $K$ \cite{Arroyo}. For a ring $R$, idempotent endomorphisms of $R$ are in a one-to-one correspondence with the pairs $(K,S)$, where $K$ is an ideal of $R$, $S$ is a subring of $R$ and $R=K\oplus S$ as abelian groups. Any such ring extension of $K$ by $S$ is completely determined by two ring morphisms $\lambda\colon S\to\End(K)$ and $\rho\colon S\to\End(K)^{\op}$, that is, by the $S$-$S$-bimodule structure on $K$. 

In \cite{BFP2} idempotent endomorphisms of left skew braces were studied, while in \cite{Facchini}  idempotent endomorphisms of pre-Lie algebras were studied. In \cite{Leila} idempotent endomorphisms were considered in relation with the study of algebras with a bilinear product.

In  \cite{Facchini}, Michela Cerqua and the second author considered, in the study of pre-Lie algebras and their modules, a notion of pre-endomorphism $M\to M'$ for arbitrary $k$-algebras $M,M'$ over a commutative ring $k$. This notion of $k$-algebra pre-morphism is a generalization of the standard notion of $k$-algebra morphism (equivalently, of the notion of Lie morphism in the case of associative $k$-algebras). Pre-morphisms can be dualized, in turn, in two different ways to the notions of generalized morphisms (which appear naturally in the study of pre-Jordan algebras) and anti-pre-morphisms (which appear in the study of anti-pre-Lie algebras \cite{Bai}).

This paper is devoted to studying idempotent pre-endomorphisms, idempotent generalized endomorphisms and idempotent anti-pre-endomorphisms. 

We begin the paper defining pre-morphisms and their related notions, and giving a motivation for their study (Sections~\ref{2} and~\ref{3}). We investigate their elementary properties, for instance giving a first and second ``Isomorphism theorem for pre-morphisms'' (Theorems~\ref{First} and~\ref{Second}). If $2$ is invertible in the base commutative ring $k$, every bilinear operation on a $k$-module $M$ decomposes in a commutative part and an anticommutative part (Theorem~\ref{4.1}). This leads us to the study of pre-morphisms and their dual generalized morphisms. In Section~\ref{5} we study the connections between these notions and superalgebras. 
Then  we determine all idempotent pre-endomorphisms of a $k$-algebra $M$ (Theorem~\ref{5.1}), all idempotent generalized endomorphisms of $M$ 
(Theorem~\ref{5.2}), and all idempotent anti-pre-endomorphisms of $M$ (Theorem~\ref{5.3}).

Finally, in Theorems~\ref{pla} and~\ref{laa}, we show that all pre-Lie algebras (all Lie-admissible algebras, respectively) with a sub-adjacent fixed Lie $k$-algebra $A$  can be also described via suitable idempotent endomorphisms of a non-associative version $T_{na}(A)$ of the tensor algebra of the $k$-module $A$.

\section{Basic facts}\label{2}

In this paper, $k$ denotes a commutative ring with identity. A {\em $k$-algebra}  is a $k$-module $_kM$ with a further $k$-bilinear operation $M\times M\rightarrow M$, $(x,y)\mapsto xy$ (or, equivalently, with a $k$-module morphism $M\otimes_kM\rightarrow M$). A {\em subalgebra} (an {\em ideal}, resp.) of $M$ is a $k$-submodule $N$ of $M$ such that $xy\in N$ for every $x,y\in N$ ($xn\in N$	and $nx\in N$ for every $x\in M$ and $n\in N$, resp.).

In \cite{Facchini}, Michela Cerqua and the second author considered pre-Lie algebras from the point of view of their multiplicative lattice of ideals, and studied their idempotent endomorphisms. Let $k$ be a commutative ring with identity. A {\em pre-Lie $k$-algebra} is a (not-necessarily associative) $k$-algebra $(M,\cdot)$ satisfying the identity
$$(x\cdot y)\cdot z - x\cdot (y\cdot z) = (y\cdot x)\cdot z - y\cdot (x\cdot z)$$
for every $x, y, z\in M$. For instance, every associative $k$-algebra is a pre-Lie algebra.

It is easy to prove that for any pre-Lie algebra $(L,\cdot)$, defining the commutator $[x, y] = xy- yx$ for every $x, y\in L$, one gets that $(L,[-,-])$ is a Lie algebra, called the Lie algebra {\em sub-adjacent} the pre-Lie algebra $(L, \cdot)$. \\

In the study of pre-Lie algebras, a natural notion of pre-morphism appears (see \cite{Facchini}). A $k$-module morphism $\varphi\colon M \rightarrow M'$, where $M,M'$ are arbitrary (not-necessarily associative) $k$-algebras, is a {\em pre-morphism} if
\begin{equation}\varphi(xy)- \varphi(x) \varphi(y) = \varphi(yx)- \varphi(y) \varphi(x)\label{1}\end{equation}
for every $x, y\in M$. For instance, given any $k$-algebra $M$, we can consider, for every element $x\in M$, the mapping $\lambda_x\colon M\to M$, defined by $\lambda_x(a)=xa$ for every $a\in M$. Let the mapping $\lambda\colon M\rightarrow \mathrm{End}(_kM)$ be defined by $\lambda\colon x\mapsto\lambda_x$ for every $x\in~M$. Then $\lambda$ is a $k$-module morphism if $M$ is an arbitrary algebra, it is a $k$-algebra morphism if and only if $M$ is associative, and it is a pre-morphism if and only if $M$ is a pre-Lie algebra \cite[Section~2]{Facchini}.

It is easy to see (\cite[Lemma 2]{Facchini}) that:
{\rm (1)} every $k$-algebra morphism is a pre-morphism;
{\rm (2)} the composite mapping of two pre-morphisms is a pre-morphism; and
{\rm (3)} the inverse mapping of a bijective pre-morphism is a pre-morphism.

\section{Extensions of endofunctors}\label{3}

There is an endofunctor $U$ of the category $\Alg_k$ of (not-necessarily associative) $k$-algebras that associates with any $k$-algebra $(A,\cdot)$ the $k$-algebra $(A,[-,-])$, where $[x,y]=xy-yx$ for every $x,y\in A$. It associates with any morphism $f\colon (A,\cdot)\to (B,\cdot)$ in $\Alg_k$, the same mapping $U(f)=f\colon (A,[-,-])\to (B,[-,-])$. But notice that for any pre-morphism $f\colon (A,\cdot)\to (B,\cdot)$, it is also possible to define $U(f):=f$, and $U(f)\colon (A,[-,-])\to (B,[-,-])$ turns out to be a morphism in $\Alg_k$. Hence $U$ can be extended to a functor $U\colon \Alg_{k,p}\to \Alg_k$ from the category $\Alg_{k,p}$ of $k$-algebras and their pre-morphisms to the category $\Alg_k$ of $k$-algebras and  $k$-algebra morphisms. Even more is true: a mapping $f\colon A\to B$ is a pre-morphism $f\colon (A,\cdot)\to (B,\cdot)$ if and only if $f\colon (A,[-,-])\to (B,[-,-])$ is a morphism in $\Alg_k$. By definition, an algebra $A$ is {\em Lie-admissible} if $U(A)$ is a Lie algebra. For every $k$-algebra $A$ the $k$-algebra $U(A)$ is always  anticommutative (i.e., $[x,y]=-[y,x]$ and $[x,x]=0$).

\medskip

What we have remarked in the previous paragraph suggests us to give some natural definitions. Let $(M,\cdot)$ be a $k$-algebra.
A {\em pre-congruence} on $(M,\cdot)$ is an equivalence relation $\sim$ on the set $M$ such that, for every $x,x',y,y'\in M$ and every $\lambda\in k$, $x\sim x'$ and $y\sim y'$ imply $x+y\sim x'+y'$, $\lambda x\sim \lambda x'$, and $[x,y]\sim [x',y']$. Correspondingly, a {\em pre-ideal} on $(M,\cdot)$ is a $k$-submodule $I$ of $_kM$ such that $[x, i]\in I$ for every $x\in M$ and every $i\in I$. (This is clearly equivalent to requiring that $[i,x]\in I$ for every $x\in M$ and every $i\in I$, because of the anticommutativity of the operation $[-,-]$).

\begin{exams} (1) One of the things that are difficult to understand for young (bright) students is the fact that the center of a group is a normal subgroup, but the center of a ring is a subring, and not an ideal. If the bright students know the definition of associated Lie algebra, the way to explain them the correct point of view is to show them that if $R$ is any associative $k$-algebra, then its center $Z(R)$ is an ideal in the Lie algebra $(R,[-,-])$. In our terminology, the same proof shows that if $(M,\cdot)$ is any Lie-admissible $k$-algebra, then the set $\{\,z\in M\mid xz=zx$ for every $x\in M\,\}$ is a pre-ideal of $(M,\cdot)$.

Assume now that $A$ is an arbitrary $k$-algebra. The {\em nucleus} $N(A)$ of $A$ is defined by $N(A) := \{\, x\in A\mid  (x,A,A) = (A,x,A) = (A,A,x) = \{0\}\, \}$, where $(-,-,-)$ denotes the associator. The nucleus is an associative subring of $A$. 
The center $Z(A)$ of the algebra $A$ is defined to be
$Z(A) = \{\, x\in N(A) \mid [x,A] = \{0\}\, \}$. Hence the center of $A$ is a pre-ideal of $N(A)$.

(2) The kernel of any pre-morphism (=the inverse image of $0$) is always a pre-ideal.\end{exams}

Clearly, for any $k$-algebra $(M,\cdot)$ there is bijection between the set of all pre-congruences on $(M,\cdot)$ and the set of all the pre-ideals of $(M,\cdot)$. It associates with any pre-congruence $\sim$ the equivalence class $[0]_\sim$ of the zero of $M$ modulo $\sim$.

Similarly to pre-ideals, it is possible to define pre-subalgebras. A {\em pre-subalgebra} of a $k$-algebra $(M,\cdot)$ is a $k$-submodule $B$ of $_kM$ such that $[x, y]\in B$ for every $x,y\in B$.
For instance, the image of any pre-morphism is a pre-subalgebra. The nucleus $N(A)$ is a pre-subalgebra. Indeed, for any $x,y\in N(A)$, we can easily show that $([x,y],c,d)=(c,[x,y],d)=(c,d,[x,y])=0$ for $c,d\in A$.

It is easy to prove that:

\begin{thm}\label{First} {\rm [First isomorphism theorem for pre-morphisms.]} Let $f\colon (M,\cdot)\to (M',\cdot)$ be a pre-morphisms. Then:
the kernel of $f$ is a pre-ideal of $(M,\cdot)$, the image of $f$ is a pre-subalgebra of $(M',\cdot)$, and there is a unique mapping $\widetilde{f}\colon M/\ker(f)\to M'$ that makes the diagram 
\[\xymatrix{M \ar[r] \ar[d]_f & M/\ker(f) \ar[ld]^{\widetilde{f}}   \\
M' & }\] commute. Moreover, $\widetilde{f}$ yields a $k$-module isomorphism $M/\ker(f)\cong f(M),$ which is a $k$-algebra isomorphism between $(M/\ker f, [-,-])$ and $(f(M), [-,-])$.
    \end{thm}

\begin{thm}\label{Second} {\rm [Second isomorphism theorem for pre-morphisms.]} Let $(M,\cdot)$ be a $k$-algebra, let $N$ be a pre-subalgebra of  $M$, and let $K$ be a pre-ideal of $M$. Then:
the sum $N+K= \{\, x+y\mid x\in N,\ y\in K\,\}$  is a pre-subalgebra of  $M$, the intersection $N\cap K$ is a pre-ideal of $N$, and the canonical $k$-module isomorphism $N/N\cap K\cong (N+K)/K$ is a $k$-algebra isomorphism $(N/N\cap K,[-,-])\cong ((N+K)/K,[-,-])$.\end{thm}

Clearly, for a $k$-algebra $(M,\cdot)$, the lattice $\Cal L(M,\cdot)$ of all ideals of $(M,\cdot)$ embeds into the lattice $\Cal L_p(M,\cdot)$ of all pre-ideals of $(M,\cdot)$, and $\Cal L_p(M,\cdot)\cong \Cal L(M,[-,-])$. The three lattices $\Cal L(M,\cdot)$, $\Cal L_p(M,\cdot)$ and $\Cal L(M,[-,-])$ are complete.

\medskip

Now recall that for a semi-abelian variety $\Cal V$ of universal
algebras, an object $X$ in $\Cal V$ and normal subalgebras $A$ and $B$ of $X$, the {\em Huq commutator} $[A,B]_H$ is the smallest normal subalgebra $C$ of $X$ such that there exists a homomorphism $\varphi\colon A\times B \to X/C$ for which both composite morphisms $\varphi\circ(1_A,0)\colon A\to X/C$ and $\varphi\circ(0,1_B)\colon B\to X/C$ are the canonical projections. All varieties of non-associative algebras are semi-abelian categories \cite[Theorem~9.5]{VdL}. In the variety of all $k$-algebras, the Huq commutator $[A,B]_H$ of two ideals $A,B$ of a $k$-algebra $X$ is the ideal of $X$ generated by the set $AB\cup BA=\{\,ab,ba\mid a\in A,\ b\in B\,\}$. For instance for our anticommutative $k$-algebra $(X,[-,-])$, the Huq commutator of two ideals $A,B$ of $(X,[-,-])$ is $[A,B]$, the $k$-submodule of $X$ generated by all products $[a,b]$, with $a\in A$ and $b\in B$.

Thus in the isomorphism $\Cal L_p(M,\cdot)\cong \Cal L(M,[-,-])$, given two pre-ideals $I,J$ of $(M,\cdot)$, the pre-ideal corresponding to the Huq commutator of the ideals $I,J$ of $(M,[-,-])$ in the lattice $\Cal L(M,[-,-])$ is the pre-ideal $[I,J]$ of $(M,\cdot)$ generated by all elements $ij-ji$, with $i\in I$ and $j\in J$. It has the following universal property:

\begin{prop}\label{3.3} Given two pre-ideals $I,J$ of $(M,\cdot)$, the pre-ideal $[I,J]$ is the smallest pre-ideal $K$ of $(M,\cdot)$ for which there is a well-defined morphism $$\varphi\colon(I\times J,[-,-])\to (M/K,[-,-])$$ such that $\varphi(i,0)=i+K$ and $\varphi(0,j)=j+K$ for every 
   $i\in I$ and $j\in J$. 
\end{prop}

The multiplicative lattices $(\Cal L_p(M,\cdot), \vee,[-,-])$ and $( \Cal L(M,[-,-]), \vee,[-,-]_H)$ are clearly isomorphic \cite{FFJ}.

Given a Mal’tsev variety with
Mal’tsev term $p(x, y, z) $, an algebra $X$ in the variety, and two congruences $\alpha$ and $\beta$, the {\em Smith commutator} $[\alpha,\beta]_S$ \cite{15} is the smallest of the congruences $\gamma$ on $X$ for which the mapping
$$p\colon \{\,(x, y, z) \mid (x, y)\in\alpha\ \mbox{\rm and }(y, z)\in\beta\,\}\to X/\gamma,$$  $$(x,y,z)\mapsto [p(x,y,z)]_\gamma,$$ is a homomorphism.

Now, the variety of all $k$-algebras and all its subvarieties are Mal’tsev varieties with Mal’tsev term $p(x,y,z)=x-y+z$.
For a $k$-algebra $(M,\cdot)$ with ideals $I,J$, the set $\{\,(y+i, y, y+j)\mid y\in M,i\in I,j\in J\,\}$ is a subalgebra of $M^3$. Thus we are looking for the smallest ideal $K$ of $M$ for which the mapping
$$p\colon \{\,(y+i, y, y+j) \mid y\in M,i\in I,j\in J\,\}\to M/K,$$  $$p(y+i, y, y+j)=y+i+j+K,$$ is a homomorphism. It is easily checked that the smallest such ideal $K$ is the ideal of $(M,\cdot)$ generated by $I\cdot J\cup J\cdot I$. Here $I\cdot J$ denotes the set of all products $i\cdot j$, and similarly for $J\cdot I$.

\medskip

Notice that Huq=Smith for any $k$-algebra $(M,\cdot)$. Given any two ideals $I,J$ of $(M,\cdot)$ their Huq=Smith commutator is the ideal of $(M,\cdot)$ generated by $I\cdot J\cup J\cdot I$.  What we have said above, in the statement of Proposition~\ref{3.3}, is that it is possible to also have a variation of the definition of the Huq commutators for pre-ideals, and the commutator of two pre-ideals is a pre-ideal, as stated in Proposition~\ref{3.3}, and it corresponds to the Huq=Smith commutator of the sub-adjacent anticommutative $k$-algebra $(M,{[-,-]})$.

\section{Dualizing}

\subsection{The first way. Jordan algebras}\label{J}

Now we want to dualize the notion of pre-morphism and the functor $U\colon \Alg_{k,p}\to \Alg_k$ we have seen above. In order to  dualize the notion of pre-morphism, we can do it in two different natural ways, which must not be confused. The first is replacing Condition (\ref{1}) with Condition \begin{equation}\varphi(xy)- \varphi(x) \varphi(y) = -(\varphi(yx)- \varphi(y) \varphi(x)),\label{3''}\end{equation} and the second is replacing Condition (\ref{1}) with Condition \begin{equation}\varphi(xy)+ \varphi(x) \varphi(y) = \varphi(yx)+\varphi(y) \varphi(x).\label{3'}\end{equation} 

\medskip 

Similarly to the endofunctor $U$ of $\Alg_k$,  there is an endofunctor $C$ of the category $\Alg_k$ that associates with any $k$-algebra $(A,\cdot)$ the $k$-algebra $(A,\circ)$, where $x\circ y=xy+yx$ for every $x,y\in A$. It associates with any morphism $f\colon (A,\cdot)\to (B,\cdot)$ in $\Alg_k$, the same mapping $C(f)=f\colon (A,\circ)\to (B,\circ)$. For every $k$-algebra $(A,\cdot)$ the $k$-algebra $(A,\circ)$ is always a commutative algebra. If the $k$-algebra $(A,\cdot)$ is associative, then the $k$-algebra $(A,\circ)$ is a Jordan algebra. \\

Let us begin with the first of these two alternatives.
\medskip

We say that a $k$-module morphism $\varphi\colon M \rightarrow M'$, where $M,M'$ are arbitrary (not-necessarily associative) $k$-algebras, is a {\em  generalized morphism} if
$$\varphi(xy)- \varphi(x) \varphi(y) = -(\varphi(yx)- \varphi(y) \varphi(x))$$
for every $x, y\in M$. For instance, $k$-algebra morphisms are generalized morphisms. It is easy to prove that a mapping $\varphi\colon M \rightarrow M'$ is a generalized morphism $(M,\cdot)\to(M',\cdot)$ if and only if it is a $k$-algebra morphism $(M,\circ)\to (M',\circ)$. In particular, the composite mapping of two generalized morphisms is a generalized morphism, and there is a category $\Alg_{k,g}$ whose objects are $k$-algebras and whose morphisms are the generalized morphisms between them. 
The endofunctor $C\colon \Alg_k\to \Alg_k$ extends to a functor $C\colon \Alg_{k,g}\to \Alg_k$. 

\bigskip

Let us briefly consider here the possibility of decomposing into a unique canonical way any $k$-bilinear operation on a $k$-module $M$ as a sum of a commutative operation and an anticommutative one.

If $M$ is any $k$-module, the set of all $k$-bilinear operations $M\times M\to M$ is a $k$-module isomorphic to the $k$-module $\Hom_k(M\otimes_kM,M)$. If $C$ is the $k$-submodule of $M\otimes_kM$ generated by the set $\{\,x\otimes y-y\otimes x\mid x,y\in M\,\}$, then the set of all commutative $k$-bilinear operations $M\times M\to M$ is a sub-$k$-module of $\Hom_k(M\otimes_kM,M)$ isomorphic to the $k$-module $\Hom_k(M\otimes_kM/C,M)$. If $A$ is the $k$-submodule of $M\otimes_kM$ generated by the set $\{\,x\otimes y+y\otimes x\mid x,y\in M\,\}$, then the set of all anticommutative $k$-bilinear operations $M\times M\to M$ is a sub-$k$-module of $\Hom_k(M\otimes_kM,M)$ isomorphic to $\Hom_k(M\otimes_kM/A,M)$.

\smallskip

In the following, $\ann_M(2)$ denotes the set of all elements $x$ of a $k$-module $_kM$ such that $x+x=0$. The module $_kM$ is $2${\em -torsion free} if $\ann_M(2)=0$.

\begin{thm}\label{4.1} Let $_kM$ be a $k$-module, $\Comm$ and $\AntiComm$ be the $k$-submodules of $\Hom_k(M\otimes_kM,M)$ consisting of all $k$-bilinear commutative and anticommutative operations on $_kM$, respectively. Then:

{\rm (a)} $\Comm\cap\AntiComm=\Hom_k(M\otimes_kM, \ann_M(2)\,)$.

    {\rm (b)} If $2$ is invertible in $k$, then $\Hom_k(M\otimes_kM,M)=\Comm\oplus\AntiComm$.

    {\rm (c)} If $k$ has characteristic $2$, then $\Comm=\AntiComm$.
\end{thm}

\begin{proof} (a) A $k$-bilinear operation is in $\Comm\cap\AntiComm$ if and only if $xy=yx=-yx$ for every $x,y$ in $M$, that is, if and only if the image of the mapping $M\otimes_kM\to M$ is contained in $\ann_M(2)$.

(b) If $2$ is invertible in $k$, then $M$ is $2$-torsion free, so that $\Comm\cap\AntiComm=0$ by (a). Hence, in order to conclude the proof of (b) it suffices to show that $\Hom_k(M\otimes_kM,M)\subseteq\Comm+\AntiComm$. Now let $\cdot$ be a $k$-bilinear operation on $M$. Let $(M,\circ)$ and $(M,[-,-])$ be the commutative algebra and the anticommutative algebra relative to the $k$-algebra $(M,\cdot)$ as constructed in Section \ref{3} and Subsection \ref{J}. Then $\cdot=\frac{1}{2}\circ+\frac{1}{2}[-,-]$. This concludes the proof of (b).
(c) is trivial.\end{proof}

Notice that if $M$ is a $k$-module, $*$ is a commutative operation on $M$, $\diamond$ is an anticommutative operation on $M$, and we define $\cdot$ as the sum $*+\diamond$, (i.e., $x\cdot y=x*y+x\diamond y$), then $\circ=2*$ and $[-,-]=2\diamond$, because $x*y+x\diamond y+y*x+y\diamond x=2(x*y)$ and $x*y+x\diamond y-y*x-y\diamond x=2(x\diamond y)$.

If $2$ is invertible in $k$, then the $k$-algebras $(M,\cdot)$ and $(M,2\cdot)$ are always isomorphic:

\begin{prop} Suppose $2$ invertible in $k$. Let $(M,\cdot)$ be a $k$-algebra. Then the $k$-algebras $(M,\cdot)$ and $(M,2\cdot)$ are isomorphic.
\end{prop}

\begin{proof} The isomorphism is the mapping $\varphi\colon (M,\cdot)\to (M,2\cdot)$, $\varphi\colon x\mapsto \frac{x}{2}$. For this mapping, both $\varphi(x,y)$
{\color{blue}  (should it be $\varphi(xy)$)}
and $\varphi(x)\varphi(y)$ are equal to $\frac{xy}{2}$.
\end{proof}

The following lemma will be useful in the sequel:

\begin{lem}\label{xxx} Let $M,M'$ be $k$-algebras with $M'$ $2$-torsion free. Then a $k$-module morphism $\varphi\colon M\to M'$ is a $k$-algebra morphism if and only if it is a pre-morphism and a generalized morphism.\end{lem}

\begin{proof} We must show that 
$\varphi(xy)= \varphi(x) \varphi(y)$ for every $x, y\in M$ if and only if
\begin{equation}
    \varphi(xy)- \varphi(x) \varphi(y) = \varphi(yx)- \varphi(y) \varphi(x)\label{1'}\end{equation}
for every $x, y\in M$ and \begin{equation}\varphi(xy)- \varphi(x) \varphi(y) = -(\varphi(yx)- \varphi(y) \varphi(x))\label{2'}\end{equation} for every $x, y\in M$. The ``only if'' implication has a trivial proof. For the ``if'' implication, if the two conditions (\ref{1'}) and (\ref{2'}) hold, then $$ \varphi(yx)- \varphi(y) \varphi(x) = -(\varphi(yx)- \varphi(y) \varphi(x))$$ for every $x, y\in M$, so that $ \varphi(yx)- \varphi(y) \varphi(x)=0$ because $M'$ is $2$-torsion free.
Therefore $\varphi$ is a $k$-algebra morphism.   
\end{proof}

\subsection{The second way. Anti-pre-Lie algebras}\label{apl}

Let us pass to the second possible way of dualizing the notion of pre-morphism and the functor $U\colon \Alg_{k,p}\to \Alg_k$. We say that a $k$-module morphism $\varphi\colon M \rightarrow M'$, where $M,M'$ are arbitrary $k$-algebras, is an {\em anti-pre-morphism} if \begin{equation}\varphi(xy)+\varphi(x)\varphi(y)=\varphi(yx)+\varphi(y)\varphi(x)\label{a}\end{equation} for every $x,y\in M$. It is easily seen that $\varphi\colon M \rightarrow M'$ is an anti-pre-morphism if and only if its opposite $-\varphi\colon M \rightarrow M'$ is a pre-morphism. In particular, if $\varphi$ is a $k$-algebra pre-morphism, then $-\varphi$ is an anti-pre-morphism, the composite mapping of two anti-pre-morphisms is a pre-morphism, and the inverse of a bijective anti-pre-morphism is an anti-pre-morphism.

Anti-pre-morphisms appear naturally in connection to the study of the recent notion of anti-pre-Lie algebra \cite{Facchini, Bai}. Let $k$ be a commutative ring with identity and $(A,\cdot)$ be a $k$-algebra. As usual, define $[x,y]:=x\cdot y-y\cdot x$ for every $x,y\in A$. The $k$-algebra $A$ is an {\em anti-pre-Lie $k$-algebra} if 
\begin{equation}
			(x\cdot y)\cdot z+x\cdot(y\cdot z)=(y\cdot x)\cdot z+y\cdot(x\cdot z)\label{3210}
		\end{equation} and \begin{equation}
		[x,y]\cdot z +[y,z]\cdot x +[z,x]\cdot y =0\label{3211}
		\end{equation} 
		for every $x,y,z\in A$ (see \cite{Bai}).

\medskip

The first notions about pre-Lie algebras can be adapted to anti-pre-Lie algebras with several non-trivial adaptations \cite{Bai}. Thus for an anti-pre-Lie $k$-algebra $(A,\cdot)$, the {\em Lie algebra sub-adjacent} to $A$ is the Lie algebra $(A,[-,-])$. That is, anti-pre-Lie algebras are also Lie-admissible, like pre-Lie algebras. More precisely, a $k$-algebra is anti-pre-Lie if and only if (1) it is Lie-admissible and (2) the mapping $\lambda\colon(A,\cdot)\to(\End(A_k),\circ)$ induces a Lie antihomomorphism $\lambda\colon(A,[-,-])\to\mathfrak g\mathfrak l(A)$, i.e., a Lie algebra morphism $\lambda\colon(A,[-,-])^{\op}\to\mathfrak g\mathfrak l(A)$, or a right module structure on the Lie algebra $(A,[-,-])$ sub-adjacent the anti-pre-Lie algebra $(A,\cdot)$.
This occurs because identity (\ref{3210}) can be rewritten as $\lambda_{x\cdot y}+\lambda_x\circ\lambda_y=\lambda_{y\cdot x}+\lambda_y\circ\lambda_x$, or equivalently as $\lambda_{[x,y]}=[\lambda_y,\lambda_x]$. Condition~(2) is equivalent to ``the mapping $\lambda\colon(A,\cdot)\to(\End(A_k),\circ)$ is an anti-pre-morphism''.

\bigskip

A mapping  $\varphi\colon M \rightarrow M'$ is an anti-pre-morphism if and only if it is a $k$-algebra morphism $(M,[-,-])\to (M',[-,-]^{\op})$. Let $\Alg_{k,a}$ be the category whose objects are $k$-algebras, whose morphisms are  anti-pre-morphisms, and composition $*$ is defined, for every anti-pre-morphism $\varphi\colon M\to M'$ and $\psi\colon M'\to M''$, by $\psi*\varphi:=-(\psi\circ\varphi)$. Then $k$-algebras with anti-pre-morphisms and this modified composition $*$ form a category $\Alg_{k,a}$. The identity morphism $M\to M$ of every object $M$ in the category $\Alg_{k,a}$ is the mapping $x\mapsto-x$.
There is a covariant functor $D\colon \Alg_{k,a}\to\Alg_k$, defined by $(M,\cdot)\mapsto(M,[-,-])$ and $f\mapsto-f$.

 If $M$ is a (not-necessarily associative) $k$-algebra,  then $M$ is an anti-pre-Lie algebra if and only if (1) $M$ is Lie-admissible, and (2)
the mapping $\lambda\colon M\rightarrow \End(_kM)$, where $\lambda\colon x\mapsto\lambda_x$ and $\lambda_x(a)=xa$,  is an  anti-pre-morphism.

\medskip

A mapping $$f\colon (A,\cdot)\to (B,\cdot)$$ is an anti-pre-morphism if and only if $$-f\colon (A,[-,-])\to (B,[-,-])$$ is a morphism in $\Alg_k$. Every $k$-algebra anti-homomorphism, (that is,  a $k$-linear mapping $f\colon M\to M'$ such that $f(xy)=f(y)f(x)$) is an anti-pre-morphism. In particular, if $M,M'$ are commutative, then every $k$-algebra morphism $M\to M'$ is an anti-pre-morphism. For every $k$-algebra morphism $f\colon M\to M'$ ($M,M'$ arbitrary $k$-algebras), the mapping $-f\colon M\to M'$ is an anti-pre-morphism. Hence it is possible to extend the category $\Alg_k$ to the category whose objects are $k$-algebras and whose morphisms are anti-morphisms, and the functor $D$ to a functor of this category to the category $\Alg_k$.

\begin{prop} There is a category isomorphism $\Alg_{k,p}\cong\Alg_{k,a}$.\end{prop}

\begin{proof} The functor $F\colon \Alg_{k,p}\to\Alg_{k,a}$, $F\colon(M,\cdot)\mapsto (M,\cdot)$, $F\colon f\mapsto -f$, is a category isomorphism. Notice that $D\circ F=U$.\end{proof}

\medskip

Notice that the composite mapping of two anti-pre-morphisms is a pre-morphism, and not an anti-pre-morphism in general. 

\begin{rem} There is a clear relation between the notions introduced until now in this paper and the concepts of Lie derivation and Jordan derivation for an associative algebra. Lie derivations and Jordan derivations for associative algebras were introduced and studied by Ancochea \cite{[1]}, Jacobson \cite{Jac}, Herstein  \cite{HerJordan, HerLie}, Bre\v{s}ar \cite{BreLie, BreJordan}, and several other mathematicians in the past decades. A Lie derivation of an associative algebra $(M,\cdot)$ is a mapping $(M,\cdot)\to (M,\cdot)$ that is a derivation $(M,[-,-])\to (M,[-,-])$ of the Lie algebra $(M,[-,-])$. Similarly for Jordan derivations of an associative algebra $(M,\cdot)$, where the Lie algebra $(M,[-,-])$ is replaced by the Jordan algebra $(M,\circ)$ in the definition. Thus, for an associative algebra, our pre-derivations are exactly Lie derivations \cite[Lemma~5]{Facchini}.

Similarly, consider the notion of generalized derivation as it was introduced in \cite{Bredistance}. In that paper, a {\em generalized derivation} $f\colon (M,\cdot)\to (M,\cdot)$ of an associative $k$-algebra  $(M,\cdot)$ is a $k$-module morphism for which there exists a derivation $d\colon M\to M$ such that $f(xy) = f(x)y +xd(y)$ for every $x,y\in M$. This is equivalent to the existence of a $k$-module morphism $d\colon M\to M$ for which $$\left\{\begin{array}{l}f(xy) = f(x)y +xd(y)\\  d(xy) = d(x)y +xd(y)\end{array}\right.$$  for every $x,y\in M$. Equivalently, if and only if there exists  a derivation $d\colon M\to M$ such that $ f(x)y -f(xy) =d(x)y-d(xy) $ for every $x,y\in M$. This equation is very similar to other equations in this paper, in the sense that the expression $f(x)y-f(xy)$ can be also seen as a sort of associator $(f,x,y)$.

If we want to be also bolder, we can say that, for any $k$-algebra $M$, the $k$-algebra $\End(_kM)$ of all endomorphisms of the $k$-module $_kM$ is an associative $k$-algebra, and $M$ is an $\End(_kM)$-$\End(_kM)$-bimodule over this associative $k$-algebra if we set $f\cdot x=f(x)$ and $ x\cdot f=0$ for every $x\in M$ and $f\in \End(_kM)$. Then a $k$-module morphism $f\in \End(_kM)$ is: (1) a right $M$-module morphism if and only if $(f,x,y)=0$ for every $x,y\in M$ (same definition as for an associative algebra), (2) a derivation if and only if $(f,x,y)=(x,f,y)$ for every $x,y\in M$ (same definition as for a pre-Lie algebra), and (3) a generalized derivation if and only if there exists a derivation $d$ of $M$ for which $(f,x,y)=(d,x,y)$ for every $x,y\in M$.

We prefer not to use the terminology Lie homomorphism, Jordan homomorphism, Lie derivation, Lie ideal, and Lie subalgebra as in \cite[p.~7]{Breetal} and \cite{HerLie} for the simple reason that if $(M,\cdot)$ is a non-associative algebra, then $(M,[-,-])$ and $(M,\circ)$ are not a Lie algebra and a Jordan algebra respectively, but only an anticommutative algebra and a commutative algebra respectively. Correspondingly, we prefer to use the terms pre-Lie-morphism, generalized morphism, pre-derivation, pre-ideal, and pre-subalgebra, though we understand the possible problematic with this terminology.
    \end{rem}

\section{Pre, generalized, anti-pre, and superalgebras}\label{5}

There are at least two natural ways of associating to any $k$-algebra $M$ a superalgebra, that is, a $\Z_2$-graded algebra. Given any $k$-algebra $M$, we can construct the $k$-module  $M\oplus M$ (the $k$-module direct sum of two copies of $M$) and define on it either the multiplication $$(a,b)(a',b'):=(aa'+bb', ab'+a'b),$$ or the multiplication $$(a,b)(a',b'):=(aa'-bb', ab'+a'b),$$ for all $a,b,a',b'\in M$.

\begin{rem} For this, we have been inspired by \cite[Section~2]{Leila}. Given any associative $k$-algebra $R$, we can consider the abelian $k$-algebra $R_a$, that is, the $k$-module $R$ with zero multiplication. Then $R_a$ is an $R$-$R$-bimodule, so that it is possible to construct the extension of $R_a$ by $R$. Since $R_a\otimes_RR_a\cong R_a$ as an  $R$-$R$-bimodule, the ``compatible bilinear forms'' $R_a\times R_a\to R$ are all of the form $(b,b')\mapsto \mu bb'$ for some $\mu\in k$. In the first paragraph of this section we have considered the most important cases of $\mu=1$ and $\mu=-1$.\end{rem}

In order to give a less repetitive presentation, we will consider the case of a functor $F_{\mu,\lambda}$ that associates with any $k$-algebra $M$ the
 superalgebra $F_{\mu,\lambda}(M):=M\oplus M$ with multiplication $$(a,b)(a',b'):=(aa'+\mu bb', ab'+a'b)$$ for all $a,b,a',b'\in M$, and with any mapping $\varphi\colon M\to M'$ the mapping $(\varphi\oplus\lambda\varphi)\colon F_{\mu,\lambda}(M)=M\oplus M\to F_{\mu,\lambda}(M')=M'\oplus M'$. Here $\mu,\lambda\in k$. For the superalgebra $F_{\mu,\lambda}:=M\oplus M$, the even subalgebra is $M\oplus 0$, and the odd part is $0\oplus M$.
The mapping $(\varphi\oplus\lambda\varphi)$ is a $k$-module morphism of degree $0$. 

The proof of the following proposition is elementary:

\begin{prop}\label{5.2'} Suppose $\mu,\lambda\in k$ and $\mu(\lambda^2-1)=0$. For any $k$-module morphism $\varphi\colon M\to M'$, we have that:

{\rm (a)} $(\varphi\oplus\lambda\varphi)$ is a pre-morphism if and only if $\varphi$ is a pre-morphism;
    
{\rm (b)} $(\varphi\oplus\lambda\varphi)$ is a generalized morphism if and only if $\varphi$ is a generalized morphism;

{\rm (c)} $(\varphi\oplus\lambda\varphi)$ is an anti-pre-morphism if and only if $\varphi$ is an anti-pre-morphism.
\end{prop}
Thus, for every $\mu,\lambda\in k$ and $\mu(\lambda^2-1)=0$, we have three endofunctors $F_{\mu,\lambda}$ of the categories $\Alg_{k,p}$, $\Alg_{k,g}$ and $\Alg_{k,a}$ respectively, which associate with any $k$-algebra $M$ the $\Z_2$-graded algebra $F_{\mu,\lambda}(M)=M\oplus M$. 

\begin{rem} The most interesting cases are those in which $\mu$ and $\lambda$ are equal to $1$ or $-1$. If the characteristic of $k$ is different from $2$, these four cases ``$\mu$ and $\lambda$ are equal to $1$ or $-1$'' yield four different endofunctors $F_{\mu,\lambda}$. If $k$ has characteristic $2$, these four endofunctors $F_{\mu,\lambda}$ coincide.\end{rem}

Another very interesting case is that of $\mu\in k$ arbitrary, and $\lambda=-1$. In this case, Proposition~\ref{5.2'} becomes:

\begin{cor} Let $\mu$ be an arbitrary element of $k$. Then there are three endofunctors $F_{\mu,-1}$ of the three categories $\Alg_{k,p}$, $\Alg_{k,g}$ and $\Alg_{k,a}$. They associate with any $k$-algebra $M$ the
 superalgebra $F_{\mu,-1}(M):=M\oplus M$ with multiplication $$(a,b)(a',b'):=(aa'+\mu bb', ab'+a'b)$$ for all $a,b,a',b'\in M$, and with any morphism $\varphi\colon M\to M'$ the morphism $$(\varphi\oplus-\varphi)\colon F_{\mu,-1}(M)=M\oplus M\to F_{\mu,-1}(M')=M'\oplus M'.$$ Moreover:

{\rm (a)} if $\varphi$ is a pre-morphism, then $-\varphi$ is an anti-pre-morphism and $(\varphi\oplus-\varphi)$ is a pre-morphism;
    
{\rm (b)} if $\varphi$ is an anti-pre-morphism, then $-\varphi$ is a  pre-morphism and $(\varphi\oplus-\varphi)$ is an anti-pre-morphism.
\end{cor} 

Let us go back to the general case of two elements $\mu,\lambda$ of $k$ satisfying $\mu(\lambda^2-1)=0$, and to the case of the endofunctor $F_{\mu,-1}\colon\Alg_{k,p}\to \Alg_{k,p}$. If $\varphi\colon M\to M'$ is a pre-morphism, then $(\varphi\oplus\lambda\varphi)\colon M\oplus M\to M'\oplus M'$ is a pre-morphism, hence it induces a $k$-algebra morphism $(\varphi\oplus\lambda\varphi)\colon (M\oplus M,[-,-])\to (M'\oplus M',[-,-])$, where the commutator $[-,-]$ is defined like in all the previous pages. But notice that for superalgebras, another commutator is also usually considered, called that {\em supercommutator}, and defined on homogeneous elements by $[x,y]_s=xy-(-1)^{|x|\cdot|y|}yx$.
Hence, for any two elements $(a,b)(a',b')\in M\oplus M$, we have that \begin{equation}\begin{array}{l}[(a,b),(a',b')]_s=[(a,0),(a',0)]_s+[(a,0),(0,b')]_s+[(0,b),(a',0)]_s+[(0,b),(0,b')]_s= \\ \qquad =(aa'-a'a,0)+(0,ab'-b'a)+(0,ba'-a'b)+(\mu bb'+\mu b'b, 0)= \\ \qquad =([a,a']+\mu(b\circ b'), [a,b']+[b,a']).\end{array}\label{s}\end{equation}

This formula suggests us to consider the case $\mu=1$ and to  examine the mapping $\varepsilon\colon M\to M\oplus M$ defined by $\varepsilon(x)=(x, x)$ for every $x\in M$:

\begin{prop} Suppose $\mu=1$, and let $\varepsilon\colon M\to M\oplus M$ be the mapping defined by $\varepsilon(x)=(x, x)$ for every $x\in M$. Then $$[\varphi(x),\varphi(y)]_s=(2xy, 2[x,y])$$ for every $x,y\in M$.\end{prop}

\begin{proof}
  $[\varphi(x),\varphi(y)]_s=[(x,x),(y,y)]_s=([x,y]+\mu(x\circ y), [x,y]+[x,y])$ by Equation~(\ref{s}). Thus $[\varphi(x),\varphi(y)]_s=([x,y]+x\circ y, 2[x,y])=(2xy, 2[x,y]).$
\end{proof}

We conclude this section with the following proposition. Notice that if $M,M'$ are $k$-algebras and $\varphi\colon M\to M'$ is a $k$-algebra morphism, then $\varphi\colon (M,\cdot)\to(M',\cdot)$ is a pre-morphism if and only if $(\varphi\oplus\lambda\varphi)\colon (M\oplus M,[-,-])\to (M'\colon M',[-,-])$ is a $k$-algebra morphism. In the next proposition we show that we have a completely different result replacing the multiplication $[-,-]$ with the multiplication $[-,-]_s$.

\begin{prop} Suppose $\mu,\lambda\in k$ and $\lambda^2=1$. Let $M,M'$ be $k$-algebras with $M'$ $2$-torsion free. Then a $k$-module morphism $\varphi\colon M\to M'$ is a $k$-algebra morphism $\varphi\colon (M,\cdot)\to(M',\cdot)$ if and only if $(\varphi\oplus\lambda\varphi)\colon (M\oplus M,[-,-]_s)\to (M'\colon M',[-,-]_s)$ is a $k$-algebra morphism.\end{prop}

\begin{proof}
  The mapping $(\varphi\oplus\lambda\varphi)\colon (M\oplus M,[-,-]_s)\to (M'\colon M',[-,-]_s)$ is a $k$-algebra morphism if and only if $(\varphi\oplus\lambda\varphi)([(a,b),(a',b')]_s)=[(\varphi(a),\lambda\varphi(b)), (\varphi(a'),\lambda\varphi(b'))]_s$  for every $(a,b),(a',b')\in M\oplus M$. Applying Identity (\ref{s}), this can be re-written as $(\varphi([a,a'])+\mu\varphi(b\circ b'), \lambda\varphi([a,b'])+\lambda\varphi([b,a']))=([\varphi(a),\varphi(a')]+\mu((\lambda\varphi(b))\circ \lambda\varphi(b')), [\varphi(a),\lambda\varphi(b')]+[\lambda\varphi(b),\varphi(a')])$.
  
  Since $\lambda^2=1$, this is equivalent to the two identities $\varphi([a,a'])=[\varphi(a),\varphi(a')]$ and $\varphi(b\circ b')=\varphi(b))\circ \varphi(b')$, which mean that $\varphi$ is both a pre-morphism and a generalized morphism. By Lemma~\ref{xxx}, this is equivalent to ``$\varphi$ is a $k$-algebra morphism".
\end{proof}

\section{Idempotent pre-endomorphisms}

Making use of the functor $U$ described in the first paragraph of Section~\ref{3}, it is possible to describe idempotent pre-endomorphisms of a $k$-algebra $(M,\cdot)$.

Here by idempotent pre-endomorphism $e\colon M\rightarrow M$ of a $k$-algebra $M$ we mean a $k$-linear mapping such that $e^2=e$ and 
 \begin{equation}e(xy)- e(x) e(y) = e(yx)- e(y) e(x)\end{equation}
for every $x,y\in M$. Recall that it is possible to associate with any $k$-algebra $(M,\cdot)$ the commutative $k$-algebra $(M,\circ)$ and the anticommutative $k$-algebra $(M,[-,-])$. 

\begin{thm}\label{5.1} Let $M$ be a $k$-algebra. There is a bijection between the set $E:=\{\, e\in \End_k(M)\mid e$ is a pre-morphism and $e\colon M\to M$ is idempotent$\,\}$ of all idempotent pre-endomorphisms of $M$ and the set $P$ of all pairs $(K,B)$, where $K$ is a pre-ideal of $M$, $B$ is a pre-$k$-subalgebra of $M$, and $_kM=K\oplus B$ as a $k$-module. The pair corresponding to a pre-endomorphism $e\in E$ is the pair $(\ker(f), f(M))$. Conversely, the idempotent pre-endomorphism that corresponds to a pair $(K,B)\in P$ is the composite mapping of the second canonical projection $\pi_2\colon {} _kM=K\oplus B\to B$ and the inclusion $\varepsilon_2\colon B\hookrightarrow {}_kM$.
    \end{thm} 

\begin{proof} Let us prove the last part of the statement, i.e., let us show that $\varepsilon_2\pi_2\colon (M,\cdot)\to (M,\cdot)$ is a pre-morphism. Equivalently, we must show that $\varepsilon_2\pi_2\colon (M,[-])\to (M,[-])$ is a morphism. This is true, because both $\pi_2\colon (M,[-])\to (B,[-])$ and $\varepsilon_2\colon (B,[-])\to (M,[-])$ are morphisms. For instance, $\pi_2$ is a morphism because $\pi_2([m,m'])=\pi_2([k+b,k'+b'])=\pi_2([k,k'])+\pi_2([k,b'])+\pi_2([b,k'])+\pi_2([b,b'])=[b,b']=[\pi_2(m),\pi_2(m')]$.
\end{proof}

Like in the case of Lie algebras, where we have the hierarchy $$\mbox{\rm associative algebras}{}\Rightarrow{}\mbox{\rm pre-Lie algebras}{}\Rightarrow{}\mbox{\rm Lie-admissible algebras}{}\Rightarrow{}\mbox{\rm algebras,}$$ for Jordan algebras we also have a hierarchy $$\mbox{\rm associative algebras}{}\Rightarrow{}\mbox{\rm pre-Jordan algebras}{}\Rightarrow{}\mbox{\rm Jordan admissible algebras}{}\Rightarrow{}\mbox{\rm algebras.}$$ Here a {\em pre-Jordan $k$-algebra} \cite{admin} is a $k$-algebra $(A,\cdot)$ in which \begin{equation*} (x \circ y)\cdot (z\cdot u) + (y\circ z)\cdot (x\cdot u) + (z\circ x)\cdot (y\cdot u)
= z\cdot [(x\circ y)\cdot u] + x\cdot [(y\circ z)\cdot u] + y\cdot [(z\circ x)\cdot u]\end{equation*} and \begin{equation*}
x\cdot [y\cdot (z\cdot u)] + z\cdot [y\cdot (x\cdot u)] + [(x\circ z)\circ y]\cdot u
= z\cdot [(x\circ y)\cdot u] + x\cdot [(y\circ z)\cdot u] + y\cdot [(z\circ x)\cdot u]\end{equation*} for every $x,y,z,u\in A$,
where $x\circ y = x\cdot y + y\cdot x$. A {\em Jordan-admissible algebra} \cite{Albert (1948)} is a $k$-algebra $(A,\cdot)$  that becomes a Jordan algebra under the product $x\circ y = x\cdot y + y\cdot x$. 

It is now easy to dualize the previous results. Let $(M,\cdot)$ be any $k$-algebra.
Define a {\em generalized congruence} on $(M,\cdot)$ as an equivalence relation $\sim$ on the set $M$ such that, for every $x,x',y,y'\in M$ and every $\lambda\in k$, $x\sim x'$ and $y\sim y'$ imply $x+y\sim x'+y'$, $\lambda x\sim \lambda x'$, and $x\circ y\sim x'\circ y'$. Correspondingly, a {\em generalized ideal} on $(M,\cdot)$ is a $k$-submodule $I$ of $_kM$ such that $x\circ i\in I$ for every $x\in M$ and every $i\in I$. For instance, the
kernel of any generalized morphism (that is, the inverse image of $0$) is always a generalized ideal. Define a {\em generalized subalgebra} of a $k$-algebra $(M,\cdot)$ as a $k$-submodule $B$ of $_kM$ such that $x\circ y\in B$ for every $x,y\in B$.
 
 We have already seen that, similarly to the endofunctor $U$ of $\Alg_k$,  there is an endofunctor $C$ of the category $\Alg_k$ that associates with any $k$-algebra $(A,\cdot)$ the $k$-algebra $(A,\circ)$, where $x\circ y=xy+yx$ for every $x,y\in A$. This endofunctor $C\colon \Alg_k\to \Alg_k$ extends to a functor $C\colon \Alg_{k,g}\to \Alg_k$. 

Recall the definition of  generalized morphism, and that a mapping $\varphi\colon M \rightarrow M'$ is a generalized morphism $(M,\cdot)\to(M',\cdot)$ if and only if it is a $k$-algebra morphism $(M,\circ)\to (M',\circ)$. It is then easy to prove that:

\begin{thm}\label{5.2} Let $M$ be a $k$-algebra. There is a bijection between the set $E:=\{\, e\in \End_k(M)\mid e$ is a generalized morphism and $e\colon M\to M$ is idempotent$\,\}$ of all idempotent generalized endomorphisms of $M$ and the set $P$ of all pairs $(K,B)$, where $K$ is a generalized ideal of $M$, $B$ is a generalized $k$-subalgebra of $M$, and $_kM=K\oplus B$ as a $k$-module. The pair corresponding to a generalized endomorphism $e\in E$ is the pair $(\ker(f), f(M))$. Conversely, the idempotent generalized endomorphism that corresponds to a pair $(K,B)\in P$ is the composite mapping of the second canonical projection $\pi_2\colon {} _kM=K\oplus B\to B$ and the inclusion $\varepsilon_2\colon B\hookrightarrow {}_kM$.
    \end{thm} 

In this dualization, we have made use of the ``first way'' of dualizing (Subsection~\ref{J}). Let us see what occurs dualizing the ``second way'' (Subsection~\ref{apl}):

\begin{thm}\label{5.3} Let $M$ be a $k$-algebra. Suppose $M$ $2$-torsion free. Then there is a bijection between the set $E:=\{\, e\in \End_k(M)\mid e$ is an anti-pre-morphism and $e\colon M\to M$ is idempotent$\,\}$ of all idempotent anti-pre-endomorphisms of $M$ and the set $P$ of all pairs $(K,B)$, where $K,B$ are $k$-submodules of $_kM$, $_kM=K\oplus B$ as a $k$-module, $xy-yx\in K$ for every $x,y\in M$, and $bb'=b'b$ for all $b,b'\in B$. The pair corresponding to an anti-pre-endomorphism $e\in E$ is the pair $(\ker(f), f(M))$.
    \end{thm} 

\begin{proof} Let $\varphi\colon M\to M$ be an idempotent anti-pre-endomorphism. Then $\varphi=\varphi^2$, where $\varphi$ is an anti-pre-morphism and $\varphi^2$ is a pre-morphism. Therefore $\varphi$ is at the same time a pre-morphisms and an anti-pre-morphism. The two conditions (\ref{1}) and (\ref{a}) together are equivalent to $$\varphi(xy)-\varphi(yx)= \varphi(x) \varphi(y) - \varphi(y) \varphi(x)=-(\varphi(x) \varphi(y) - \varphi(y) \varphi(x)).$$ Since $M$ is $2$-torsion free, this is equivalent to $$\varphi(xy)-\varphi(yx)= \varphi(x) \varphi(y) - \varphi(y) \varphi(x)=0.$$ That is, $$\varphi(xy)=\varphi(yx)\quad\mbox{\rm and}\quad  \varphi(x) \varphi(y) = \varphi(y) \varphi(x).$$ We have shown that an idempotent anti-pre-endomorphism $\varphi\colon M\to M$ is exactly an idempotent $k$-module endomorphism $\varphi\colon M\to M$ for which $xy-yx\in \ker(\varphi)$ for every $x,y\in M$ and $bb'=b'b$ for all $b,b'\in \varphi(M)$. Now it is easy to conclude.\end{proof}

\begin{rem} For an associative $k$-algebra $R$, the most standard example of a pre-derivation that is not a derivation is constructed taking a derivation $d\colon R\to R$ and a $k$-linear mapping $\tau\colon R\to Z(R) $ such that $\tau([R,R])=0$. Then $d+\tau$ turns out to be a pre-derivation which is not necessarily a derivation. These pre-derivations are often called {\em proper} pre-derivations of the associative $k$-algebra $R$. It is a classical result by Martindale \cite{III} that for primitive rings with a non-trivial idempotent and characteristic $\ne 2$, all pre-derivations are essentially of this form.

Now it follows from Theorem~\ref{5.3} that, for any $2$-torsion free $k$-algebra $M$, idempotent anti-pre-endomorphisms of $M$ are exactly idempotent $k$-module endomorphisms $\tau\colon {}_kM\to {}_kM$ of $_kM$  such that $\tau([M,M])=0$ and $bb'=b'b$ for all $b,b'\in\tau(M)$. The relation between this result and the results stated in the previous paragraph is evident.\end{rem} 

\section{Another description of pre-Lie algebras}

Now we want to show how, for a fixed Lie algebra $A$, all pre-Lie algebras with sub-adjacent Lie algebra $A$ and all Lie-admissible algebras with sub-adjacent Lie algebra $A$ can be described, again making use of suitable idempotent endomorphisms. To this end, we need a modification of the non-associative unital tensor algebra of a $k$-module $R$, where $k$ is a commutative  ring with identity, studied in \cite[Section~3]{Leila}.  Fix an indeterminate (a symbol) $x$ and consider all non-associative monomials (formal products of elements equal to $x$ retaining parentheses), with at least one occurrence of $x$. The product of monomials $u,v$ is just $(u)(v)$. In this way we get a magma, the {\em free cyclic magma} $S$. Every element $s$ of $S$ has a {\em degree} $d(s)$, which is a positive integer: it is the number of occurrences equal to $x$ in $s$. Thus $S$ is the set of all non-associative monomials of degree $\ge 1$ in the indeterminate $x$. 

\begin{rem}{\rm Another way to describe the elements of $S$ is making use of rooted trees. The unique monomial $x$ of degree $1$ can be described by the rooted graph with one vertex. The other monomials of degree $>1$ can be described by a triple $(T,r,\le)$, where $(T,r)$ is a rooted tree, all vertices of $T$ have degree $\le 3$, the root $r$ is the unique vertex of degree $2$, and there is a linear order $\le$ on the set of all the vertices of degree $1$ (the leaves).  These rooted trees must be also considered up to isomorphism $\varphi\colon (T,r,\le)\to (T',r',\le)$, where $\varphi\colon T\to T'$ is a graph isomorphism and $\varphi$ induces an ordered set morphism between the linearly ordered set of all leaves of $T$ and the linearly ordered set of all leaves of $T'$. The product $(T,r,\le)\cdot (T',r',\le)$ is the rooted tree $(T'',r'',\le)$, where the set of vertices of $T''$ is the disjoint union of the set of vertices of $T$ and the set of vertices of $T'$, plus a further vertex $r''$. The set of edges of $T''$ is the disjoint union of the set of edges of $T$ and the set of edges of $T'$, plus two further edges, one from $r''$ to $r$ and one from $r''$ to $r'$.
Finally, the set of leaves of $T''$ is the disjoint union of the set of all leaves of $T$ and the set of all leaves of $T'$, and  the linearly ordered set of leaves of $T$  is written``on the left'' of the linearly ordered set of leaves of $T'$, so that the order of the disjoint union is such that  every leaf of $T$ is smaller than every leaf of $T'$.  The degree of the monomial is equal to the number of leaves in the corresponding rooted tree.}\end{rem}

The {\em non-associative tensor algebra} of a $k$-module $R$, where $k$ is a commutative unital ring, is $$T_{na}(R):=\bigoplus_{s\in S}R^{\otimes d(s)} .$$  In $T_{na}(R)$, we have one copy of $R^{\otimes n} $ (tensor product over $k$ of $n$ copies of the $k$-module $R$) for each element $s\in S$ of degree $n$, i.e., for each way of writing the parentheses in a monomial in $x$ of degree $n$. The algebra $T_{na}(R)$ is graded by $S$, because $(R^{\otimes d(s)})\cdot (R^{\otimes d(t)})\subseteq (R^{\otimes d((s)(t))})$.

The non-associative $k$-algebra $T_{na}(R)$ ($k$ a commutative unital ring, $R$ a $k$-module) has the universal property for non-associative $k$-algebras: Any $k$-module morphism $ R\to A$ from $R$ to 
a $k$-algebra $A $ can be uniquely extended to a $k$-algebra morphism $T_{na}(R)\to A$. 

Now let $(A,\cdot)$ be a pre-Lie $k$-algebra.  Apply the universal property of $T_{na}(A)$ to the identity morphism $A\to A$, getting a $k$-algebra surjective morphism $\varphi\colon T_{na}(A)\to A$. Let $I$ be the kernel of $\varphi$. Then $T_{na}(A)=A\oplus I$ as $k$-modules. Moreover $I$ contains all elements of $T_{na}(A)$ of the form $x\otimes y-y\otimes x-[x,y]$ and those of the form $(x\otimes y)\otimes z-(y\otimes x)\otimes z-x\otimes (y\otimes z)+y\otimes (x\otimes z)$ with $x,y,z\in A$. Notice that here $x\otimes y-y\otimes x$ is a homogeneous element of degree $2$, $[x,y]$  is a homogeneous element of degree $1$, $(x\otimes y)\otimes z-(y\otimes x)\otimes z$ is a homogeneous element of degree $3$ in the homogeneous component $A^{\otimes d(s)}=A^{\otimes 3}$ with $s=(xx)x$, and $-x\otimes (y\otimes z)+y\otimes (x\otimes z)$  is a homogeneous element of degree $3$ in the homogeneous component $A^{\otimes d(s)}=A^{\otimes 3}$ with $s=x(xx)$. It is now easy to see that:

\begin{thm}\label{pla} Let $A$ be a Lie $k$-algebra. Then the pre-Lie algebras with sub-adjacent Lie algebra $A$ are exactly the $k$-algebras isomorphic to $T_{na}(A)/I$, where $T_{na}(A)$ is the non-associative tensor algebra of the $k$-module $A$, $I$ is an ideal of $T_{na}(A)$ such that $T_{na}(A)=A\oplus I$ as $k$-modules, and $I$ contains all elements of $T_{na}(A)$ of the form $x\otimes y-y\otimes x-[x,y]$ and those of the form $(x\otimes y)\otimes z-(y\otimes x)\otimes z-x\otimes (y\otimes z)+y\otimes (x\otimes z)$ with $x,y,z\in A$. \end{thm}

Similarly for Lie-admissible algebras:

\begin{thm}\label{laa} Let $A$ be a Lie $k$-algebra. Then the Lie-admissible $k$-algebras with sub-adjacent Lie algebra $A$ are exactly the $k$-algebras isomorphic to $T_{na}(A)/I$, where $T_{na}(A)$ is the non-associative tensor algebra of the $k$-module $A$, $I$ is an ideal of $T_{na}(A)$ such that $T_{na}(A)=A\oplus I$ as $k$-modules, and $I$ contains all elements of $T_{na}(A)$ of the form $x\otimes y-y\otimes x-[x,y]$ and those of the form $(x, y, z) + (y, z, x) + (z, x, y) = (y, x, z)  + (x, z, y)+ (z, y, x)$ with $x,y,z\in A$. Here $(x, y, z) = (x\otimes y)\otimes z -x\otimes (y\otimes z)$, where $(x\otimes y)\otimes z$ is a homogeneous element of degree $3$ in the homogeneous component $A^{\otimes d(s)}=A^3$ with $s=(xx)x$, and $x\otimes (y\otimes z)$  is a homogeneous element of degree $3$ in the homogeneous component $R^{\otimes d(s)}=R^3$ with $s=x(xx)$. \end{thm}


Hence in this case we also have a one-to-one bijection between a set of suitable idempotent $k$-algebra endomorphisms of $T_{na}(A)$ and the set of all pre-Lie $k$-algebras (Lie-admissible $k$-algebras) with sub-adjacent Lie algebra $A$.

\end{document}